\documentclass[12pt,a4paper,reqno]{amsart}

\usepackage[latin1]{inputenc}
\usepackage{amssymb,amsmath,amstext,amsthm,amsfonts}
\usepackage{graphicx}
\usepackage{comment}
\usepackage{multicol}
\usepackage{subfigure}
\usepackage{hyperref}
\usepackage{tikz-cd}
\usepackage{calrsfs}
\DeclareMathAlphabet{\pazocal}{OMS}{zplm}{m}{n}

\numberwithin{equation}{section}

\usepackage{a4wide}
\usepackage{xcolor}

\newcommand{\R}{\mathbb{R}}

\newcommand{\Lt}{\mathcal{L}_t}
\newcommand{\Ls}{\mathcal{L}_s}
\newcommand{\tent}{\phi}

\newcommand{\normi}[1]{{\left\vert\kern-0.25ex\left\vert\kern-0.25ex\left\vert #1 
    \right\vert\kern-0.25ex\right\vert\kern-0.25ex\right\vert}}

\newcommand{\qand}{\quad\text{and}\quad}

\def\inte{\operatorname{int}}

\def\dist{\operatorname{dist}}

\def\id{\operatorname{id}}

\newtheorem{maintheorem}{Theorem}

\newcommand{\cmt}{\begin{maintheorem}}
\newcommand{\fmt}{\end{maintheorem}}

\newtheorem{maincorollary}[maintheorem]{Corollary}

\newcommand{\cmc}{\begin{maincorollary}}
\newcommand{\fmc}{\end{maincorollary}}

\newtheorem{lemma}{Lemma}[section]

\newtheorem{proposition}[lemma]{Proposition}

\theoremstyle{remark}
\newtheorem{remark}[lemma]{Remark}

\makeatletter
\@namedef{subjclassname@2020}{\textup{2020} Mathematics Subject Classification}
\makeatother

\thanks{JFA and OE are partially supported by     CMUP 
(UIDB/00144/2020)
and PTDC/MAT-PUR/4048/2021, which are funded by FCT (Portugal) with national (MEC) and European structural funds through the programs COMPTE and FEDER, under the partnership agreement PT2020.
}

\begin{document}
\title[Modulus of continuity of invariant densities and entropies]{Modulus of continuity of invariant densities and entropies for piecewise expanding maps}
\date{}

\author[J. F. Alves]{Jos\'{e} F. Alves}
\address{Jos\'{e} F. Alves\\ Centro de Matem\'{a}tica da Universidade do Porto\\ Rua do Campo Alegre 687\\ 4169-007 Porto\\ Portugal.}
\email{jfalves@fc.up.pt} \urladdr{http://www.fc.up.pt/cmup/jfalves}

\author[O. Etubi]{Odaudu  Etubi}
\address{Odaudu  Etubi\\
Centro de Matem\'{a}tica da Universidade do Porto\\ Rua do Campo Alegre 687\\ 4169-007 Porto\\ Portugal.}
\email{etubiodaudu@gmail.com} 

\maketitle




\begin{abstract}
Using a  perturbation result  established by Galatolo and Lucena, we obtain quantitative estimates on the continuity of the invariant densities and entropies of the physical measures for some families of piecewise expanding maps. We apply these results to a family of two-dimensional tent maps. 
\end{abstract}


\tableofcontents

\section{Introduction}
In the theory of Dynamical Systems, it is common to encounter examples of systems with simple governing laws that exhibit highly complex and unpredictable behaviour. Prominent examples include one-dimensional quadratic maps, two-dimensional H\'enon quadratic diffeomorphisms, and the Lorenz system of quadratic differential equations in three-dimensional Euclidean space. Despite their straightforward formulation, these systems display intricate dynamical properties that have inspired significant mathematical developments over the past few decades.

One of the central areas of investigation in the theory of Dynamical Systems has been the study of invariant measures that describe the statistical behavior of these systems over time. Among these, \emph{Sinai-Ruelle-Bowen (SRB) measures}, also known as \emph{physical measures}, play a pivotal role. SRB measures are particularly important because they provide a way to understand the asymptotic distribution of orbits for a wide range of initial conditions, typically those that are of full measure with respect to the Lebesgue (volume) measure on the phase space. In other words, they allow us to describe the long-term statistical behavior of almost all initial points in a given region, making them essential tools for analysing chaotic systems where individual trajectories may be unpredictable, but the statistical distribution of orbits remains stable.


A key area of research in this field involves understanding the conditions under which SRB measures exist, as well as their robustness to perturbations in the system. The continuous dependence of these measures on the underlying dynamics is particularly significant, as it reflects the stability of the system's statistical properties in response to small changes. This is not only important from a theoretical standpoint but also has practical implications, as small perturbations often arise in real-world systems due to noise or external influences. In particular, the continuity of the \emph{metric entropy} -- a quantity that measures the complexity and unpredictability of the system -- associated with SRB measures is a subject of extensive study. Metric entropy quantifies the rate of information production in the system and is directly related to the degree of chaos present.

Many results in the literature have been devoted to the study of SRB measures and their associated metric entropies in various dynamical settings as  \cite{A00, ABV00, ADLP17, APP09, BC85, BC91, BY92, BY93a, B75, BR75, GB89, J81, LY73, LY85a, LY85b, P77, R76, S72}. These results range from classical systems like hyperbolic maps and diffeomorphisms to more complex systems involving non-uniformly hyperbolic dynamics or systems with piecewise smooth structures. Understanding how SRB measures and their entropies behave under perturbations provides valuable insights into the stability and predictability of chaotic systems, a topic that has been extensively explored in works such as \cite{ACF10, ACF10b, AOT06, AP21, APV17, AS14, AV02, BR19, F05, GL20, K82}. This ongoing research continues to shed light on the delicate interplay between deterministic chaos and statistical regularity, offering a deeper comprehension of the fundamental nature of chaotic dynamical systems.


In this work, we focus on \emph{absolutely continuous invariant probability measures} (\emph{ACIP}s), a class of measures that are absolutely continuous with respect to the Lebesgue measure, and often coincide with SRB measures in the context of non-uniformly hyperbolic systems. Specifically, we present estimates  on the modulus of continuity of the densities and metric entropies of ergodic ACIPs for certain classes of piecewise expanding maps in any finite dimension. These results extend the conclusions of previous works \cite{AP21,APV17}  and provide a deeper understanding of how the statistical properties of such systems behave under perturbations. To achieve this, we employ a result of Galatolo and Lucena in~\cite{GL20}, which enables us to establish   the modulus of continuity mentioned above.

Our primary application of these theoretical results is focused on a particular family of two-dimensional tent maps introduced  in \cite{PRT14}. This family is especially interesting because it is related to limit return maps that arise when a homoclinic tangency is unfolded by a family of three-dimensional diffeomorphisms, as discussed in \cite{PRT14,T01a}.
In previous works, the existence of ergodic ACIPs for these tent maps was established in \cite{PRT15}, and the continuity of the densities of these measures, along with their entropies, was demonstrated in \cite{AP21,APV17}. Building on these foundational results, we now strengthen the previous conclusions by showing that the densities and   metric entropies associated with these ACIPs perturbed by size $\delta$, has a modulus of continuity  of order $\delta \log \delta$.  
Our results are quite sharp for somewhat uniform families of piecewise expanding maps; see \cite{B07,BS08}.
We remark that  just H\"older continuity could also be deduced from Keller-Liverani results in \cite{KL99}.

The continuity of physical measures and their metric entropies plays a fundamental role in understanding the stability of dynamical systems, particularly those exhibiting complex or chaotic behavior.  By addressing the modulus of continuity of both physical measures and their metric entropies, our work aims to provide deeper insights into the intricate relationship between the dynamics of a system and the statistical properties of its invariant measures. Understanding this relationship is key to predicting the robustness of complex dynamical systems under perturbations, a question that has far-reaching implications in various fields, from mathematical theory to applied sciences. This research sheds light on the subtle and quantitative ways in which physical measures and their metric entropy responds  to changes in the underlying system, offering a more nuanced understanding of the stability and variability of chaotic systems. Through this investigation, we contribute to the broader goal of characterizing the resilience of dynamical systems to fluctuations, and thus advancing the overall theory of dynamical stability.

\subsection{Modulus of continuity}\label{se.pebd}
Here we present the general setting under which our main results will be obtained. Let $\Omega$ be a compact subset of~$\R^d$, for some $d\ge 1$. Consider   $m$ the \emph{Lebesgue}  measure on $\Omega$ and, for each $1\le p\le\infty$, the respective space $L^p(\Omega)$  endowed with its usual norm $\|\quad\|_p$.  Absolute continuity will be always meant with respect to~$m$. 
Let  $(\phi_t)_{t\in I}$    be a family of  transformations  $\phi_t: \Omega \to \Omega$, where $I$ is a metric space.
We assume that there exists $N\in\mathbb N\cup\{\infty\}$ and, for each $t\in I$, there exists an $m$ mod~0 partition $\{ R_{t,i}\}_{i=1}^N$ 
  of $\Omega$ such that each $R_{t,i} $  is a
closed domain with piecewise $C^2$ boundary of finite
$(d-1)$-dimensional measure. We also assume  that  
\begin{equation}\label{eq.maps}
\phi_{t,i}=\phi_t|_{R_{t,i}}
\end{equation}
 is a $C^2$ bijection from  $\inte(R_{t,i})$, the interior  of~$R_{t,i}$, onto its image, with a $C^2$ extension to the boundary of~$R_{t,i}$. Consider the \emph{Jacobian} function
 $$J_{t}=|\det (D\phi_t)|,$$
 defined on the (full Lebesgue measure)  subset of points in~$\Omega$ where $\phi_t$ is differentiable. 
Next we state some conditions  for our family of maps.
\begin{itemize}
\item[(P1)] there exists $\sigma_t>0$ such that for all $1\le i< N$ and all $x\in\inte(\phi_t(R_{t,i}))$
$$\| D\phi_{t,i}^{-1}(x)\|
\le \sigma_t.$$ 
 \end{itemize}
\begin{itemize}
 \item[(P2)] there
exists $\Delta_t\ge 0$ such that for all $1\le i< N$ and all $x,y\in\inte(R_{t,i})$ 
$$
\log \frac{J_{ t}(x)}{J_{ t}(y)}\le \Delta_t\,\|\phi_t(x)-\phi_t(y)\|.
$$
 \end{itemize}
\begin{enumerate}
 \item[(P3)]   there exist   $\alpha_t,\beta_t >0$ and, for each $1\le i< N$, there exists a $C^1$
unitary vector field~$X_{t,i}$ on $\partial \phi_t(R_{t,i})$\footnote{ At the   points $x\in\partial \phi_t(R_{t,i})$
where $\partial \phi_t(R_{t,i})$ is not smooth the vector $X_{t,i}(x)$ is a
common $C^1$ extension of $X_{t,i}$ restricted to each
$(d-1)$-dimensional smooth component of $\partial \phi_t(R_{t,i})$ having
$x$ in its boundary. The tangent space
at any such point   is the union of the tangent
spaces to the  $(d-1)$-dimensional smooth components that point 
belongs to.} such that:
\begin{enumerate}
\item[(a)]   the line segments joining  each $x\in\partial \phi_t(R_{t,i})$ to
$x+\alpha_t X_{t,i}(x)$ are pairwise disjoint, contained in $\phi_t(R_{t,i})$ and their union is a neighborhood of $\partial \phi_t(R_{t,i})$ in $\phi_t(R_{t,i})$;
\item[(b)] for each $x\in\partial \phi_t(R_{t,i})$
and $v\in T_x\partial \phi_t(R_{t,i})\setminus\{0\}$, we have $\left|\sin\angle (v,X_{t,i}(x))\right|\geq\beta_t$,
where $\angle (v,X_{t,i}(x))$ denotes the angle between $v$ and $X_{t,i}(x)$.
 \end{enumerate}
 \end{enumerate}
Under these conditions, it was established in \cite{A00} (see also \cite{GB89} for the case of a finite number of smoothness domains) that each $\phi_t$ has some ergodic absolutely continuous invariant probability measure. Assuming the uniqueness of this measure for each $t$, the continuity of the measure in relation to the parameter $t$ was also proven in \cite{APV17}, under the following uniformity condition: 
\begin{itemize}
\item[(U)]  there exists $\ell\ge1$ 
such that $\phi_t^j$ satisfies (P1)-(P3) for each $1\le j\le\ell$; moreover, 
there exist $0<\theta<1$ and  $M>0$ such that,
 for all $t\in I$ and $1\le j\le\ell$,
 \begin{equation*} 
\quad \quad\quad\sigma_{t,\ell}\left(1+\frac1{\beta_{t,\ell}}\right)\le \theta,\quad \sigma_{t,j}\left(1+\frac1{\beta_{t,j}}\right)\le M \qand 
\Delta_{t,j}+ \frac{1}{\alpha_{t,j}\beta_{t,j}} + 
\frac{\Delta_{t,j }}{\beta_{t,j}}   \le M,
\end{equation*}
  where $\sigma_{t,j}, \Delta_{t,j}, \alpha_{t,j},\beta_{t,j}$ are the constants in (P1)-(P3) for the map $\phi^j_t$.

\end{itemize}

To establish the modulus of continuity of these measures and their entropies, some additional conditions are required.  Set for each $s,t\in I$ and    $1\le i< N$
$$
K_{t,s,i}=\phi_{s,i}^{-1}(\phi_t(R_{t,i})\cap\phi_s(R_{s,i}))\qand \psi_{t,s,i}=\phi_{t,i}^{-1}\circ \phi_{s,i}|_{K_{t,s,i}}.$$
Naturally, we consider $\psi_{t,s,i}$ only when $K_{t,s,i} \neq \emptyset$. In fact, conditions (1)-(3) of Theorem~\ref{density} below  essentially mean that  the sets $\phi_t(R_{t,i})$ and $\phi_s(R_{s,i})$ are close to each other and  the maps $\phi_{t,i}$ and $\phi_{s,i}$ are also close to each other. 
Let $\id$ represent the identity map on $\mathbb{R}^d$, possibly restricted to some subset of $\mathbb{R}^d$.
 Define the difference set
 $$A=\{s-t: s,t\in I\}.
 $$
 Given a compact set $K\subset\mathbb R^d$ and a function   $\psi:K\to\mathbb R^d$,  let
 $$\|\psi\|_0=\sup_{x\in K}\|\psi(x)\|,$$
 where $\|\quad\|$ denotes the Euclidean norm in $\mathbb R^d$. In what follows we   assume $0\log0=0$.
 
 \begin{maintheorem}\label{density} 
Let  $(\phi_t)_{t\in I}$ be a family of   maps for which (U) holds  and    each $\phi_t $ has a unique ergodic absolutely continuous invariant probability measure $\mu_t$.
Assume that 
 there exists a function $\mathcal E:A\to\mathbb [0,1)$   such that, for all $s,t\in I$, 
  \begin{enumerate}
    \item  $\displaystyle \sum_{i=1}^Nm\left(\phi_{t,i}^{-1}\left(\phi_t(R_{t,i})\setminus\phi_s(R_{s,i})\right)\right)^{1/d}\le \mathcal E(t-s);$
    \item $\displaystyle\sum_{i=1}^N \|\psi_{t,s,i}-\id\|_0 \le \mathcal E(t-s)$; 
 \item $\displaystyle
 \displaystyle\sum_{i=1}^N\left|\frac{J_s}{J_t\circ \psi_{t,s,i}}-1\right|\le \mathcal E(t-s)$.
\end{enumerate}
%
%
 Then, there exist $C>0$ and $0<\eta<1$ such that, for all $s,t\in I$,
 $$\left\|\frac{d{\mu_t}}{dm}-\frac{d{\mu_{s}}}{dm}\right\|_1\le C\mathcal E(t-s)|\log \mathcal E(t-s)|.$$
 \end{maintheorem}

 The factor $1/d$ in assumption (1) is related to an application of  Sobolev and   H\"older inequalities in the proof of Proposition~\ref{Tnorm}. Notice that, in the case that $N$ is finite, we just need the bounds for the summands.

 Under the assumptions of Theorem~\ref{density}, using Keller-Liverani stability result, we may obtain a quantitative  estimate on the continuity of the resolvent operators of the transfer operators associated with our family of maps $(\phi_t)_{t\in I}$; see Remark~\ref{re.KL} below.

For each $t\in I$, let $h_{\mu_t}(\phi_t)$ denote the \emph{entropy} of the transformation $\phi_t$ with respect to the $\phi_t$-invariant   measure~$\mu_t$.

  \begin{maintheorem}\label{entropy} 
Let  $(\phi_t)_{t\in I}$ be a family of  
  maps  for which (U) holds  and  each $\phi_t$ has a unique absolutely continuous invariant probability measure $\mu_t$. 
Assume that 
\begin{enumerate}
\item there exists a function $\mathcal E: A\to \mathbb R^+$   such that, for all $s,t\in I$,
 $$\left\|\log {J_s}-\log {J_t}\right\|_d\le \mathcal E(t-s)\qand\left\|\frac{d{\mu_t}}{dm}-\frac{d{\mu_{s}}}{dm}\right\|_1\le \mathcal E(t-s);$$
 \item $h_{\mu_t}(\phi_t)=\int_\Omega \log J_t \,dm$, and
 there is $M>0$ such that
 $\|\log J_t\|_\infty \le M,$ for all $t\in I$.
 \end{enumerate}
 Then, there exists  some constant $C>0$ such that, for all $s,t\in I$,
 $$\left|{h_{\mu_t}(\phi_t)}-{h_{\mu_s}(\phi_s)}\right|\le C\mathcal E(t-s).$$
 \end{maintheorem}
 
Conditions for the validity of an  entropy formula as in assumption (2) of Theorem~\ref{entropy}   were obtained in \cite{AM23,AOT06,AP21}.

\subsection{Two-dimensional tent maps}\label{tent.ap} We apply the previous theorems to a  family of two-dimensional piecewise expanding   maps  introduced in~\cite{PRT14}. 
 Consider the triangle  $\Omega\subset \mathbb R^2$, which is the union of the two triangles
\begin{equation*}\label{algo}
R_1=\{(x_1,x_2):0 \leq x_1 \leq 1, \ 0 \leq x_2 \leq x_1 \}
\end{equation*}
and
\begin{equation*}
R_2=\{(x_1,x_2):1 \leq x_1 \leq 2, \ 0 \leq x_2 \leq 2-x_1 \}.
\end{equation*}
Consider the map  $\tent_{1}:\Omega\to \Omega$, given by
\begin{equation*}
\tent_1(x_1,x_2)= \left\{
\begin{array}{ll}
( x_1+x_2 , x_1-x_2 ), & \mbox{if }    (x_1,x_2)\in R_1;\\
( 2-x_1+x_2 , 2-x_1-x_2 ), & \mbox{if } (x_1,x_2)\in R_2.%
\end{array}
\right.
\end{equation*}
The \emph{tent maps} $\phi_t:\Omega\to \Omega$ are defined for      $0<t\leq 1$ by
   \begin{equation}\label{familyt2}
\phi_t=t\phi_1.
\end{equation}
Note that $R_1$ and $R_2$ are the smoothness domains of $\phi_t$, separated  by  the common straight line segment
$
\mathcal C=\{(x_1,x_2)\in \Omega: x_1=1\}.
$
These  tent maps can be described geometrically as follows: first the triangle $\Omega$ is folded through $\mathcal C$, making $R_2$ overlap $R_1$; then a flip of this domain is made and expanded to $\Omega$, thus obtaining   $\phi_1(\Omega)$; for the other maps $\phi_t$, we apply a final contraction by the factor $t$.
 \begin{figure}[h]
 \begin{center}
 \includegraphics[width=16cm]{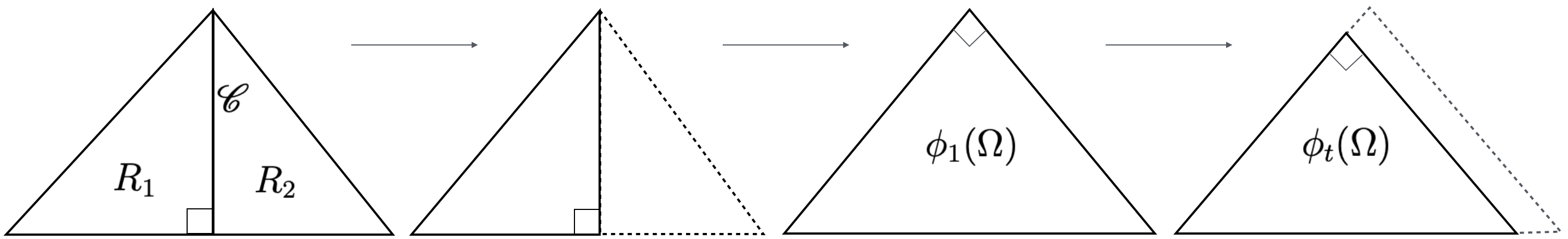}
 \caption{The tent maps}
 \end{center}
 \end{figure}

\noindent  
It was proved  in \cite{PRT15} that, for each $t \in [\tau,1]$, with  
\begin{equation}\label{eq.tau}
\tau=\frac{1}{\sqrt{2}}(\sqrt{2}+1)^{1/4},
\end{equation}
the map $\tent_t $ exhibits a \emph{strange attractor} in $ \Omega$, thereby extending the results obtained in~\cite{P06} only for $t=1$. The existence and uniqueness of an ergodic  absolutely continuous $\phi_t$-invariant probability measure $\mu_t$   was obtained in~\cite{PRT15}, for each $t \in [\tau,1]$. In the next result we improve the conclusions of~\cite{AP21,APV17}, on the continuity of the measures and their respective entropies for this family of maps. This   in particular implies that  these quantities vary   H\"older continuously.


\begin{maintheorem}\label{th.tent}
 There exists $C>0$ and $0 < \eta < 1$ such that, for all   $s,t\in [\tau,1]$,   
 $$\left\|\frac{d{\mu_t}}{dm}-\frac{d{\mu_{s}}}{dm}\right\|_1\le C|t-s|\cdot|\log |t-s||\qand \vert {h_{\mu_t}(\phi_t)}-{h_{\mu_s}(\phi_s)}\vert\le C\mathcal |t-s|\cdot|\log |t-s|| .$$
\end{maintheorem}

Results by Baladi and Smania give that linear response   fails for one-dimensional tent maps under some transversality condition of the topological class; see \cite{B07,BS08}. However, a recent result by Bahsoun and  Galatolo shows that linear response holds if one replaces the critical point in the one-dimensional map by a singularity; see~\cite{BG24}. It would be interesting to check whether changing the dimension of the system has an effect on linear response or not. In particular, it would be interesting to check if linear response holds, or not, within the higher dimensional family of tent maps that we consider in this article.

As a consequence of Theorem~\ref{th.tent}, we get a H\"older continuity estimate for   the dependence on the parameters of the resolvents of the transfer operators associated with the family of tent maps $(\phi_t)_{t\in [\tau,1]}$; see Remark~\ref{re.KL} below.

\addtocontents{toc}{\protect\setcounter{tocdepth}{1}}

\subsection*{Acknowledgement}  The authors are grateful to Wael Bahsoun for several insightful discussions and careful reading of an early draft of this work. The authors also thank the anonymous referee for pointing out   the possibility of improving the conclusion in Theorem~\ref{density} using~\cite{GL20} instead of \cite{KL99}.

\addtocontents{toc}{\protect\setcounter{tocdepth}{2}}

 \section{Functions of bounded variation}\label{variation}

The main ingredient for the proof of the above theorems is the
notion of variation for functions in multidimensional spaces. We
adopt the definition presented  in \cite{G84}. Given $f\in
L^1(\R^d)$ with compact support,
 we define the {\it variation} of $f$ as
$$V(f)=\sup\left\{\int_{\R^d}f\mbox{div}(g)dm\,:\,g\in
C_0^1(\R^d,\R^d)\text{ and }  \|g\|\leq 1\right\},$$ where
$C_0^1(\R^d,\R^d)$ is the set of $C^1$ functions from
$\R^d$ to $\R^d$ with compact support, $\mbox{div}(g)$ is the divergence of $g$  and $\|\quad\|$ is
the sup norm in $C_0^1(\R^d,\R^d)$. Integration by parts gives that
if $f$ is a $C^1$ function with compact support, then 
\begin{equation}\label{varc1}
V(f)=\int_{\R^d}\|Df\|dm.
\end{equation}
%
%
%
%
We shall use the following  properties of bounded variation functions whose proofs 
 may be
found in \cite{G84}, respectively in Remark 2.14,  Theorem 1.17 and Theorem 1.28.
 \begin{itemize}
\item[(B1)] If $f\in BV(\mathbb R^d)$ is zero outside a compact domain $K$ whose boundary is Lipschitz continuous, $f|_K$ is continuous and $f|_{\inte (K)}$ is $C^1$, then
$$V(f)= \int_{\inte(K)} \|Df\| dm   + \int_{\partial K}|f| d \bar m ,
$$
where $\bar m $ denotes the $(d-1)$-dimensional measure on $\partial K$.
\item[(B2)]  Given $f\in BV(\mathbb R^d)$, there is a sequence $(f_n)_n$ of $C^\infty$
 maps such that
 $$
 \lim_{n\rightarrow\infty}\int|f-f_n|dm=0
 \qand
 \lim_{n\rightarrow\infty}\int\|Df_n\|dm=V(f).
 $$
 \item[(B3)] There is some constant $C>0$  such that, for any $f\in
BV(\mathbb R^d)$,
 \begin{equation}\label{bvp}
 \left(\int|f|^pdm\right)^{1/p}\leq C\,V(f),
\quad \mbox{with}\quad p=\frac{d}{d-1}.
\end{equation}
 \end{itemize}
 This last property is known as \emph{Sobolev Inequality}. Notice that     $p=d/(d-1)$ is the conjugate of $d\ge  1$, meaning that
  \begin{equation}\label{eq.conjugate}
\frac1p+\frac1d=1.
\end{equation}
In the next lemma we obtain a general fact about bounded variation functions that plays a key role in this work.
  
\begin{lemma}\label{AV}
If $K$ is a compact subset of $\mathbb R^d$ and $\psi:K\to\R^d$ is a diffeomorphism onto its image, then there exists  $C>0$ such that, for all $f\in BV(\mathbb R^d)$,
$$\int_K|f\circ\psi-f|dm\le C\|\psi-\id\|_0 V(f).
$$
\end{lemma}
\begin{proof}
We start by proving the result for a continuous
piecewise affine function  $f$. More precisely, suppose that the
support $\Delta$ of $f$ can be decomposed into a finite number
of domains $\Delta_1,\dots,\Delta_N$ such that the gradient
$\nabla f$ of $f$ is a constant  vector $\nabla_i f$ on each $\Delta_i$. 
Using (B1), we obtain
 $$
\int_K|f\circ \psi-f|dm\leq  
\int_K \|\psi-\id\|_0\cdot\|\nabla f\| dm
\le \|\psi-\id\|_0 V(f).
 $$
 
 The next step is to deduce the result for any $C^1$ function $f$.
 For this, we take a sequence $(f_n)_n$ of continuous piecewise
 affine functions  such that
 $$
 \|f-f_n\|_0\rightarrow 0\qand \|Df-Df_n\|_0\rightarrow 0,
 \quad\mbox{when}\quad n\rightarrow\infty
 $$
(the derivatives $Df_n$ are defined only in the interior of
the smoothness domains). 
Then, using~\eqref{varc1} and dominated convergence theorem, we have
 $$
 V(f)=\int\|Df\|dm=\lim_{n\rightarrow\infty}\int\|Df_n\|dm
 =\lim_{n\rightarrow\infty}V(f_n)
 $$
 and
 $$
 \int_K|f\circ\psi-f|dm=\lim_{n\rightarrow\infty}
 \int_K|f_n\circ\psi-f_n|dm.
 $$
Using the   case already seen, we get the conclusion   also for $f$.

 For the general case, we know by (B2)
 that given $f\in BV(\mathbb R^d)$ there is a sequence $(f_n)_n$
 of $C^1$ maps for which
 \begin{equation}\label{e.approx}
 \lim_{n\rightarrow\infty}\int|f-f_n|dm=0
 \qand
 \lim_{n\rightarrow\infty}V(f_n)=V(f).
 \end{equation}
 We have
 $$
 \int_K
 |f\circ\psi-f|dm\leq\int_K|f\circ\psi-f_n\circ\psi|dm
 +\int_K|f_n\circ\psi-f_n|dm
 +\int_K|f_n-f|dm.
 $$
Taking $\rho=1/|\det D\psi|\circ \psi^{-1}$, we may write
 $$
 \int_K|f_n\circ
 \psi-f\circ\psi|dm=\int_{\psi(K)}|f_n-f|\cdot \rho dm
 \leq \|\rho\|_0\int|f_n-f|dm.
 $$
 The conclusion 
in this case  follows from
 (\ref{e.approx}) and the previous case.
\end{proof}

\section{Transfer operators}\label{se.transfer}
 
 Let   $\Omega\subset\R^d$ be the common domain of the  maps in the family $\{\phi_t\}_{t\in I}$.  
 For each $t\in I$, consider  the  \emph{transfer operator}
$$\Lt:L^1(\Omega)\longrightarrow L^1(\Omega),$$ defined for each $f\in L^1(\Omega)$ by
$$\Lt
f=\sum_{i=1}^{N}\frac{f\circ\phi_{t,i}^{-1}}{J_t\circ\phi_{t,i}^{-1}}
\chi_{\phi_t(R_{t,i})},$$
where $\{R_{t,i}\}_{i=1}^N$ are  the  domains of smoothness of   $\phi_t:\Omega\to\Omega$ and $\phi_{t,i}$ are the maps introduced in~\eqref{eq.maps}.
 It is well known that the
following   properties hold for each $\Lt$:

\begin{itemize} 
\item[(C1)] for all $f,g$ for which the integrals make sense, we have
$$\int_\Omega f\Lt g \,dm =\int_\Omega f\circ \phi_t\, g\, dm;$$
\item[(C2)]  $|\Lt f|\le \Lt(|f|)$ and  $\|\Lt f\|_1\leq \| f\|_1$, for
all $f\in L^1(\Omega)$;
\item[(C3)] $f \in L^1(\Omega) $ is the density of an absolutely continuous $\phi_t$-invariant   measure if and only if $f \ge 0$  and $\Lt f=f$.
\end{itemize}

Next, we study the action of the transfer operators on   the space of {\em bounded variation} functions in
 $\Omega$,
$$BV(\Omega)=\left\{ f\in L^1(\Omega):V(f)<+\infty\right\}.$$ 
Property (B3) gives in particular $BV(\Omega)\subset L^p(\Omega)$, for some $p>1$. 
Set for each $f\in BV(\Omega)$  $$\|f\|_{BV}=\|f\|_1+V(f).$$
It is   well known   that this defines a norm, and  $BV(\Omega)$  endowed with this norm becomes a Banach space; see e.g. \cite[Remark 1.12]{G84}.

\begin{proposition}
\label{LY}
Under   assumption (U), there exist $0<\lambda<1$ and  $C>0$ such that, for all $t\in I$, $f\in BV(\Omega)$ and $n\ge 1$, we have
$$\|\Lt^n f\|_{BV}\le C \lambda^n\|f\|_{BV}+ C \|f\|_1.$$
\end{proposition}

\begin{proof} 
Take $\ell\ge 1$ as in    (U).
It is a standard fact  that $\Lt^\ell$ is the transfer operator for $\phi_t^\ell$.
By  \cite[Lemma 5.4]{A00},  we have for any    $f\in BV(\Omega)$   
 \begin{equation}\label{LYL0}
V(\mathcal L^\ell_t f)\le \sigma_{t,\ell}\left(1+\frac1{\beta_{t,\ell}}\right)V(f)+M\|f\|_1\le \theta  V(f)+ M \|f\|_1,
\end{equation}
 and so
  \begin{equation}\label{LYL}
\|\mathcal L^\ell_t f\|_{BV}\le \theta V(f)+ (M+1) \|f\|_1 .
\end{equation}
Given $n\ge1$, consider $q\ge0$ and $0\le r<\ell$ such that $n=\ell q+r$. 
It follows from (C2) and~\eqref{LYL0} that
\begin{eqnarray*}
V(\mathcal L^{\ell q}_t f)&\le &\theta V(\mathcal L^{\ell (q-1)}_t f)+M \|f\|_1\\
&\le &\theta^2 V(\mathcal L^{\ell (q-2)}_t f)+(\theta +1)M \|f\|_1\\
&\vdots& \\
&\le &\theta^q V( f)+(\theta^{q-1}+\theta^{q-2}+\cdots +1)M \|f\|_1.
\end{eqnarray*}
 It follows   that
 \begin{equation}\label{eq.label1}
 \|\Lt^{\ell q} f\|_{BV}=V(\Lt^{\ell q} f)+\|\Lt^{\ell q} f\|_1\le \theta^{q}V(f)+ \bigg(1+M\sum_{j\ge0}\theta^j\bigg) \|f\|_1.
\end{equation}
On the other hand,  
   \begin{equation}\label{eq.label2}
   V(\mathcal L^r_t f)\le  \sigma_{t,r}\left(1+\frac1{\beta_{t,r}}\right) V(f)+ M \|f\|_1\le M V(f)+ M \|f\|_1.
   \end{equation}
Finally, using~\eqref{eq.label1} and~\eqref{eq.label2}, we get
 \begin{eqnarray*}
\|\Lt^n f\|_{BV}&=&\|\Lt^{ \ell q} \Lt^r f\|_{BV}\\
&=&\theta ^q V(\Lt^r f)+ \bigg(1+M\sum_{j\ge0}\theta^j\bigg) \|f\|_1\\
&\le & 
 \theta ^q MV(  f)+ \bigg(M+1+M\sum_{j\ge0}\theta^j\bigg) \|f\|_1.
\end{eqnarray*}
Now, observe that
 $$\theta^q=\theta^{(n-r)/\ell} = \left(\theta^{1/\ell}\right)^n \theta^{-r/\ell}
 \le \left(\theta^{1/\ell}\right)^n \theta^{-1}.
 $$
 Take $$\lambda=\theta^{1/\ell}\qand C=\max\left\{\frac M\theta, M+1+M\sum_{j\ge0}\theta^j\right\}$$ and recall that $V(f)\le \|f\|_{BV}$.
\end{proof}

It follows from the previous result  that $\Lt(BV(\Omega))\subset BV(\Omega)$. From here on, we  consider  $\Lt$ as an operator  from the space $BV(\Omega)$ into itself.
%
%

\begin{proposition}\label{Tnorm}
Under the  assumptions   of Theorem~\ref{density}, there exists  $C>0$ such that, for all $f\in BV(\Omega)$,  
$$\|\Lt f-\Ls f\|_1\le C \mathcal E(t-s)  \|f\|_{BV}.$$
\end{proposition}
\begin{proof}
Given $f\in BV(\Omega)$, we have
$$\|\Lt f-\Ls f\|_1\le C \mathcal E(t-s)  \|f\|_{BV}.$$
Indeed, 
\begin{align*}
\|\Lt f-\Ls f\|_1&\le  \sum_{i=1}^N\int_{\Omega} \left|\frac{f\circ\phi_{t,i}^{-1}}{J_t\circ\phi_{t,i}^{-1}}
\chi_{\phi_t(R_{t,i})}-\frac{f\circ\phi_{s,i}^{-1}}{J_s\circ\phi_{s,i}^{-1}}
\chi_{\phi_s(R_{s,i})}\right|dm\\
&=\underbrace{\sum_{i=1}^N\int_{\phi_t(R_{t,i})\cap\phi_s(R_{s,i})} \left|\frac{f\circ\phi_{t,i}^{-1}}{J_t\circ\phi_{t,i}^{-1}}-
\frac{f\circ\phi_{s,i}^{-1}}{J_s\circ\phi_{s,i}^{-1}}
\right|dm}_{\text{(I)}}+\\
&\quad+
\underbrace{\sum_{i=1}^N\int_{\phi_t(R_{t,i})\setminus\phi_s(R_{s,i})} \left|\frac{f\circ\phi_{t,i}^{-1}}{J_t\circ\phi_{t,i}^{-1}}\right|dm}_{\text{(II)}}
+
\underbrace{\sum_{i=1}^N\int_{\phi_s(R_{s,i})\setminus\phi_t(R_{t,i})} \left|\frac{f\circ\phi_{s,i}^{-1}}{J_s\circ\phi_{s,i}^{-1}}
 \right|dm}_{\text{(III)}}
\end{align*}
We just need to obtain the appropriate  bounds for (I), (II) and~(III).
To estimate (I), note that
  by the change of variables $y=\phi_{s,i}(x)$, we have 
\begin{equation*}\label{eq.changeofvar}
\int_{\phi_t(R_{t,i})\cap\phi_s(R_{s,i})} \left|\frac{f\circ\phi_{t,i}^{-1}}{J_t\circ\phi_{t,i}^{-1}}-\frac{f\circ\phi_{s,i}^{-1}}{J_s\circ\phi_{s,i}^{-1}}\right|dm
=\int_{\phi_{s,i}^{-1}\left(\phi_t(R_{t,i})\cap\phi_s(R_{s,i})\right)} \left|\frac{f\circ\phi_{t,i}^{-1}\circ \phi_{s,i}} {J_t\circ\phi_{t,i}^{-1}\circ \phi_{s,i}}-\frac{f}{J_s}\right|J_s\,dm.
\end{equation*}
Set $\psi_{t,s,i}=\phi_{t,i}^{-1}\circ \phi_{s,i}|_{\phi_{s,i}^{-1}\left(\phi_t(R_{t,i})\cap\phi_s(R_{s,i})\right)}$. Therefore,
\begin{align*}
\text{(I)} &= 
 \sum_{i=1}^N\int_{K_{t,s,i}}  \left|\frac{f\circ\phi_{t,i}^{-1}\circ \phi_{s,i}} {J_t\circ\phi_{t,i}^{-1}\circ \phi_{s,i}}-\frac{f}{J_s}\right|J_s\,dm\\
&\le  \sum_{i=1}^N\int_{K_{t,s,i}} |f\circ\psi_{t,s,i}-f|\left|\frac{J_s}{J_t\circ\psi_{t,s,i}}\right|\,dm  + \sum_{i=1}^N\int_{K_{t,s,i}} \left|\frac{J_s}{J_t\circ\psi_{t,s,i}}-1\right||f|\,dm.
\end{align*}
By assumption (3)  of Theorem~\ref{density},  there exists some $C_0>0$ such that, for each $1\le i<N$,  
$$
\left|\frac{J_s}{J_t\circ\psi_{t,s,i}}\right| \le 1+\sup_{1\le i<N}
\left|\frac{J_s}{J_t\circ\psi_{t,s,i}}-1\right|
  \le 1+ \sum_{i=1}^N
\left|\frac{J_s}{J_t\circ\psi_{t,s,i}}-1\right| \le C_0.
$$
By  Lemma \ref{AV} and the  assumptions   of Theorem~\ref{density},   we may write
\begin{align*}
\text{(I)} &\le C_0 \sum_{i=1}^N\|\psi_{t,s,i}-\id\|_0\, V(f)+\sum_{i=1}^N \left|\frac{J_s}{J_t\circ\psi_{t,s,i}}-1\right|\|f\|_1\\
&\le  C\mathcal E(t-s)\|f\|_{BV},
\end{align*}
for some uniform constant $C>0$.
To estimate (II), note that, by change of variables, 
\begin{align*}
\int_{\phi_t(R_{t,i})\setminus\phi_s(R_{s,i})} \left|\frac{f\circ\phi_{t,i}^{-1}}{J_t\circ\phi_{t,i}^{-1}}\right|dm =  \int_{\phi_{t,i}^{-1}(\phi_t(R_{t,i})\setminus\phi_s(R_{s,i}))}  |f| dm 
=
\int_\Omega \chi_{\phi_{t,i}^{-1}(\phi_t(R_{t,i})\setminus\phi_s(R_{s,i}))}   |  f |dm.
\label{eq.integrals}
\end{align*}
Observe that, by the Sobolev Inequality (B3), we have $f\in L^p(\Omega)$, with $p=d/(d-1)$ and~$d$ being conjugate.   It follows from H\"older and Sobolev inequalities and  assumption~(1) of Theorem~\ref{density}  that there exists some $C>0$ such that
\begin{align*}
\text{(II)}&=  \sum_{i=1}^N\int_\Omega  \chi_{\phi_{t,i}^{-1}(\phi_t(R_{t,i})\setminus\phi_s(R_{s,i}))} | f |dm\\
&\le\sum_{i=1}^N\|\chi_{\phi_{t,i}^{-1}(\phi_t(R_{t,i})\setminus\phi_s(R_{s,i}))}\|_d\|f\|_p\\
&\le C  \sum_{i=1}^N m\left(\phi_{t,i}^{-1}(\phi_t(R_{t,i})\setminus\phi_s(R_{s,i}))\right)^{1/d}\|f\|_{BV}\\
&\le C\mathcal  E(t-s) \|f\|_{BV}.
\end{align*}
We are done for (II). 
The calculations follow  similarly for (III). 
\end{proof}

\section{Modulus of continuity for the  densities}

In this section we prove Theorem~\ref{density}. We will use  a Galatolo-Lucena stability result presented  in~\cite{GL20} for uniform families of operators as  in \cite[Definition 9.1]{GL20} described below in (UF1)-(UF4). First, 
notice that 1 is an isolated eigenvalue, by the theorem of Ionescu Tulcea and Marinescu \cite{ITM50}, since $\Lt$ is quasicompact; moreover, it has multiplicity one, because we assume that each $\phi_t $ has a unique ergodic absolutely continuous invariant probability measure; recall (C3).  Let $\rho_t \in BV(\Omega)$ 
denote  the density of that measure. We need to show that there exist $C>0$ and $0<r<1$ such that, for all $t,s\in I$,
\begin{enumerate}
\item[(UF1)] $\|\rho_t\|_{BV}\le C $;
\item[(UF2)] $\|(\Lt-\Ls)\rho_t\|_1\le C\mathcal E(t-s) $;
\item[(UF3)] \emph{$\|\Lt^nf\|_1\le Cr^n\|f\|_{BV}$, for all  $f\in BV(\Omega)$ with $\int f dm=0$ and $n\ge1$;}
\item[(UF4)] \emph{$\|\Lt^n f\|_1\le   C \|f\|_1,$ for all  $f\in BV(\Omega)$ and $n\ge1$.}
\end{enumerate}
Let us now justify that these conditions are satisfied. In fact, it follows from Proposition~\ref{LY} that, for all $n\ge1$,
$$\|\Lt^n f\|_{BV}\le C\lambda^n\|f\|_{BV}+ C \|f\|_1,$$
for all $f\in BV(\Omega)$ and $t\in I$. Since $\rho_t$ is the fixed point for $\Lt$ with $\|\rho_t\|_1=1$ and the last inequality holds for all $n\ge1$,  we   get
(UF1). Then, (UF2) follows from  Proposition~\ref{Tnorm} and (UF1).
Now, it follows from~\cite[Theorem~3]{R83} that, for all $t \in I$ and  $f \in L^1(\Omega)$,  we have
$$
\Pi_t f = \rho_t \int f \,dm,$$
 where $\Pi_t$ is the projection onto the eigenspace of the eigenvalue 1 with respect to the operators $\Lt$. Observing that $\Pi_t f=0$ for all $f \in L^1(\Omega)$ with $\int f dm=0$,  condition (UF3) is  a consequence of \cite[Theorem~1]{R83}. Finally, (UF4) follows easily from (C3) in Section~\ref{se.transfer}.
%
%
%

The conclusion of Theorem~\ref{density} then follows from~\cite[Proposition~32]{GL20}.

\begin{remark}\label{re.KL}
The considerations in the beginning of this section, together with Proposition~\ref{LY} and Proposition~\ref{Tnorm}, show that we are   in the setting of  Keller-Liverani stability result. This enables us to obtain a quantitative  estimate on the continuity of the resolvents associated with the transfer operators for our family of maps. Specifically, it follows form \cite[Theorem 1]{KL99}  that, fixing $\delta>0$, $r\in(\lambda,1)$ and $\eta=\log(r/\lambda)/\log(1/\lambda)$ (with $0<\lambda<1$ given by Proposition~\ref{LY}), there are constants $a,b,c,d>0$ such that  for any $z\in\mathbb C$ with
 $$|z|>r \qand \dist(z,\sigma(\Ls))>\delta,$$
 where $\sigma(\Ls)$ is the spectrum of $\Ls$, we have for all $f\in BV(\Omega)$
 \begin{enumerate}
\item $\|(z-\Lt)^{-1}f\|_{BV}\le a\|f\|_{BV}+b\|f\|_1$;
\item $\normi{(z-\Lt)^{-1}-(z-\Ls)^{-1}}\le \mathcal E(t-s)^\eta
\left(c\|(z-\Ls)^{-1}\|_{BV}+d \|(z-\Ls)^{-1}\|_{BV}^2\right),
$
\end{enumerate}
where $\normi{\quad}$ is defined for an operator $T:BV(\Omega)\to BV(\Omega)$ by
$$\normi{T}= \sup_{\{f \in BV(\Omega) : \| f \|_{BV} \leq 1 \}} \|T f \|_1.$$
\end{remark}

\section{Modulus of continuity for  the entropies}
Here we prove Theorem~\ref{entropy}.
For each $t\in I$, let $\rho_t$ denote the density of $\mu_t$ with respect to $m$.
Since     the entropy formula in assumption (2)  holds,   we have for all $s,t\in I$
\begin{align*}
 |h_{\mu_s}(\phi_s)-h_{\mu_t}(\phi_t)|&= \left|\int\log J_sd\mu_s-\int\log J_t d\mu_t\right|\\
 &\le \left|\int\left(\log J_s-\log J_t \right)d\mu_s\right|+  \left|\int\log J_td\mu_s-\int\log J_t d\mu_t\right|\\
 &\le \left|\int\left(\log J_s- \log J_t \right)\rho_s dm\right|+  \left|\int\log J_t(\rho_s-\rho_t)dm\right|.
\end{align*}
Using   H\"older inequality and the bound in assumption (2), we get
\begin{equation}\label{eq.seven}
 \left|\int\log J_t(\rho_s-\rho_t)dm\right|\le M\|\rho_s-\rho_t\|_1.
\end{equation}
On the other hand, (UF1) gives that
$
\|\rho_s\|_{BV}\le C.
$
Taking $p=d/(d-1)$, it follows from Sobolev Inequality that there exists a constant $C'>0$ such that  
\begin{equation}\label{eq.eight}
\|\rho_s\|_p\le   C'.
\end{equation}
Hence, using H\"older Inequality and \eqref{eq.eight} we get
\begin{equation}
\left|\int\left(\log J_s - \log J_t \right)\rho_s dm\right|\le  \|\rho_s\|_p \left\|\log J_s - \log J_t\right\|_d\le C'\left\|\log J_s- \log J_t\right\|_d. \label{eq.nine}
\end{equation}
The conclusion follows from~\eqref{eq.seven},~\eqref{eq.nine} and  assumption (1) of Theorem~\ref{entropy}.

\section{Application to tent maps}
 \label{se.tents}

Here we prove Theorem~\ref{th.tent}. Our strategy is to  apply   Theorem~\ref{density} and Theorem~\ref{entropy} to the family  of tent maps $(\tent_t)_{t \in [\tau,1]}$ presented in Subsection~\ref{tent.ap}.
 We know  that $R_1$ and~$R_2$ are the only domains of smoothness of every $\phi_t$.  Therefore, for each $t\in[\tau,1]$ and $i=1,2$, we have
  $$\phi_{t,i}=\phi_t|_{R_{i}},\quad K_{t,s,i}=\phi_{s,i}^{-1}(\phi_t(R_{i})\cap\phi_s(R_{i}))\qand \psi_{t,s,i}=\phi_{t,i}^{-1}\circ \phi_{s,i}|_{K_{t,s,i}} .$$
   According to \cite[Section 4]{APV17},   the uniformity condition   (U) is satisfied with $\ell=6$. Existence and uniqueness of an ergodic  absolutely continuous $\phi_t$-invariant probability measure $\mu_t$   was obtained in~\cite{PRT15}, for all $t \in [\tau,1]$.   Moreover, 
    the entropy formula holds for this family of maps, by \cite[Theorem G]{AP21}. 
 We are left  to verify the assumptions  of Theorem~\ref{density} and Theorem~\ref{entropy} with adequate estimates to deduce Theorem~\ref{th.tent}.
It is enough  to show that
 there exists some constant $M>0$   such that, for all $s,t\in [\tau,1]$ and $i=1,2$, we have
\begin{enumerate}
\item[(a)] $\displaystyle m\left(\phi_{t,i}^{-1}\left(\phi_t(R_{i})\setminus\phi_s(R_{i})\right)\right) \le M |t-s|;$
\item[(b)] $ \|\psi_{t,s,i}-\id\|_0 \le M |t-s|;$ 
\item[(c)] $\displaystyle
\left|\frac{J_s}{J_t\circ \psi_{t,s,i}}-1\right|\le M |t-s|;$
\item[(d)] $\left\|\log{J_s}-\log{J_t}\right\|_d\le M |t-s|;$
\item[(e)] $\|\log J_t\|_\infty \le M.$
\end{enumerate}
Indeed, 
from ~\eqref{familyt2}, we easily deduce that, for all $(y_1,y_2)\in \tent_{t,1}(R_1)$, we have
  $$
 \tent_{t,1}^{-1}(y_1,y_2)=\left(\frac{1}{2t}(y_1+y_2),\frac{1}{2t}(y_1-y_2)\right) 
  $$
 and, for all $(y_1,y_2)\in \tent_{t,2}(R_2)$, we have
 $$
 \tent_{t,2}^{-1}(y_1,y_2)=\left(\frac{1}{2t}(4t-y_1-y_2),\frac{1}{2t}(y_1-y_2)\right).
 $$
%
%
Moreover, each  map $\phi_t$ is piecewise linear with 
\[
D\phi_t(x_1, x_2)=
\left(
\begin{array}{cc}
	t  & t  \\
	t  &  -t
\end{array}
\right)
\]
for all $(x_1, x_2)\in R_1\setminus\mathcal C$, and
\[
D\phi_t(x_1, x_2)=
\left(
\begin{array}{cc}
	- t  & t  \\
	- t  &  -t
\end{array}
\right)
\]
for all $(x_1, x_2)\in R_2\setminus\mathcal C$.
Therefore, we have  $$J_t=2t^2,$$ for all $(x_1, x_2)\in \Omega\setminus\mathcal C$ and $\tau\le t\le 1$.

\subsubsection*{Proof of (a)} Observe from the dynamics of $\tent_t$, that $\tent_{t,1}(R_1)=\tent_{t,2}(R_2)$ and, moreover the Jacobian of $\phi_{t,1}$ is constant and equal to the Jacobian of $\phi_{t,2}$. Therefore, it is enough to show the conclusion for $i=1$. In fact, for $t>s$ (and for $t<s$ there is nothing to be proved, since in that case $\phi_t(R_1)\subset \phi_s(R_1)$), we have
\begin{align}
	m(\phi_{t}(R_1)\setminus \phi_{s}(R_1)) &\le\text{lenght} (\phi_t(\mathcal{C}))\,\| \phi_t(1,0)-\phi_s(1,0)\|\nonumber \\
	&= \sqrt{2}t \Vert (t,t)-(s,s) \Vert = 2t(t-s). \label{eq.measure}
\end{align}
Since the Jacobian of $\phi_{t,1}$ is constant and equal to $2t^2 $, we deduce that the Jacobian of $\phi_{t,1}^{-1}$ is  $1/(2t^2) $, which together with~\eqref{eq.measure} yields
\begin{equation*}
	m\left(\phi_{t,i}^{-1}\left(\phi_t(R_{i})\setminus\phi_s(R_{i})\right)\right) \le  \frac{(t-s)}{t}
	\le    \frac{(t-s)}{\tau}.
\end{equation*}

\subsubsection*{Proof of (b)} 
For each $(x_1,x_2) \in R_1$ and $\tau\le t\le 1$, we have
\begin{align*}
    \|\psi_{t,s,i}-\id\|_0 
    &= \sup_{(x_1,x_2) \in R_1} \|\tent_{t,1}^{-1}\circ \tent_{s,1}(x_1,x_2)-(x_1,x_2)\|\\
    &=\sup_{(x_1,x_2) \in R_1}\left\|\ \left(\frac{s}{t}-1\right)(x_1,x_2)\right\|\\
    &\le \frac{\sqrt{2}}{\tau} |{t-s}|,
\end{align*}
and
for each $(x_1,x_2)\in R_2$, we have
\begin{align*}
         \|\psi_{t,s,i}-\id\|_0&= \sup_{(x_1,x_2) \in R_2} \|\tent_{t,2}^{-1}\circ \tent_{s,2}(x_1,x_2)-(x_1,x_2)\|\\
         &= \sup_{(x_1,x_2) \in R_2}\left \|\left(\frac{s}{t}-1\right)\left(x_1-2, x_2 \right) \right \|\\
         &\le \frac{\sqrt{2}}{\tau} |{t-s}|. 
\end{align*}

\subsubsection*{Proof of (c)} For all $\tau \le s,t\le 1$, 
we have
$$
\left|\frac{J_s }{J_t\circ \psi_{t,s,i}}-1\right|= \left|\frac{s^2}{t^2}-1\right| =  \left|\frac{(s-t)(s+t)}{t^2} \right|\le \frac{2}{\tau^2}|t-s|.
$$ 
%
%
%
%
%
 
\subsubsection*{Proof of (d)}  
By the mean value theorem, 
we have that for all $\tau\le s,t\le 1$
 $$\left\|\log {J_s}-\log{J_t}\right\|_2=\left(\int_\Omega\left(\frac{2}{\tau}(s-t)\right)^2 \,dm\right)^{1/2}  \le \frac2\tau m(\Omega) |t-s|.$$
 
\subsubsection*{Proof of (e)}   For all $\tau\le  t\le 1$, we have  
 $$\|\log J_t \|_\infty = |\log (2t^2)|\le  \log 2.$$
 Recall that the expression for $\tau$ in~\eqref{eq.tau} gives $1<2\tau^2\le 2t^2\le 2$.

\bibliographystyle{acm}

\end{document}